\documentclass[12pt,1p, preprint]{elsarticle}
\usepackage{newstyle}
\usepackage{xcolor}
\begin{document}

\begin{frontmatter} 
\title{A Conjecture of Flach and Morin}
\author[1]{Bruno Chiarellotto\fnref{fn1}}
\ead{chiarbru@math.unipd.it}
\author[1]{Nicola Mazzari\corref{cor1}\fnref{fn3}}
\ead{nicola.mazzari@unipd.it}
\author[1]{Yukihide Nakada\fnref{fn1,fn2}}
\ead{yukihide.nakada@unipd.it}
\cortext[cor1]{Corresponding author}
\fntext[fn1]{Supported by grant MIUR-PRIN2017 ``Geometric, Algebraic, and Analytic Methods in Arithmetic''}
\fntext[fn2]{Supported by INdAM grant INdAM-DP-COFUND-2015}
\fntext[fn3]{Supported by BIRD 2022 "Arithmetic cohomology theories"}
\affiliation[1]{organization={Dipartimento di Matematica ``Tullio Levi-Civita'', Università degli Studi di Padova},
addressline={Via Trieste 63}, 
postcode={35121}, 
city={Padova}, 
country={Italy}}
\begin{abstract}
    A conjecture recently stated by Flach and Morin relates the action of the monodromy on the Galois invariant part of the $p$-adic Beilinson--Hyodo--Kato cohomology of the generic fiber of a scheme defined over a DVR of mixed characteristic to (the cohomology of) its special fiber.
    We prove the conjecture in the case that the special fiber of the given arithmetic scheme is also a fiber of a geometric family over a curve in positive characteristic.
\end{abstract}

\end{frontmatter}

\section{Introduction}
    Fix a prime number $p$.
    Let $K/\Q_p$ be a finite extension with ring of integers $\O_K$ and residue field $k$. Let $W=W(k)$ be the ring of Witt vectors of $k$ and let $K_0$ be its fraction field.
    We write $S = \op{Spec}(W)$, and in the context of log structures we write $S^{\times}$ (resp. $S^\varnothing$) for $S$ equipped with the canonical log structure $1 \mapsto 0$ (resp. the trivial log structure).   

    Let $f: X \to \op{Spec}(\O_K)$ be a flat, projective morphism of relative dimension $n$.
    We write $X_s$ and $X_{\overline{s}}$ for its special and geometric special fiber, respectively, and $X_\eta$ and $X_{\overline{\eta}}$ for its generic and geometric generic fibers, respectively.
    In \cite[\S 7]{FlachMorin2018}, Flach and Morin speculate on the relationship between the geometric cohomology theories of the generic and special fibers.
    The geometric cohomology groups of the generic fiber 
    \[
    X\mapsto H^{B,i}_{HK}( X_{\overline{\eta},h} )
    \]
   are the \emph{Beilinson--Hyodo--Kato} ones, considered in \cite{NekovarNiziol2016} and taking values in the category of $(\phi, N, G_K)$-modules (for more on this structure see, for example, \cite{Berger2004}).
    This cohomology theory was defined by Beilinson for any $K$-scheme $Z$, neither smooth nor proper,  using $h$-descent \cite{Beilinson2013} {}\footnote{The Beilinson--Hyodo--Kato is related to $p$-adic étale cohomology via the Fontaine functor $D_{pst}$, namely 
    \[
    H^{B,i}_{HK}(X_{\overline{\eta},h}) \cong D_{pst}(H^i(X_{\overline{\eta}}, \Q_p)):=\colim_{H\le G_K, \text{open}} (\mathbb{B}_{st}\otimes_{K_0}H^i(X_{\overline{\eta}}, \Q_p))^H \ ,
    \]
    but we will not use this fact.}.

    Moreover, when the morphism $f$ is smooth or log-smooth, there are canonical isomorphisms between Beilinson--Hyodo--Kato cohomology and the cohomology of the geometric special fiber: namely, if $f$ is smooth then it coincides with crystalline cohomology of the geometric special fiber $X_{\overline{s}}$ and if $f$ is log-smooth then it coincides with the log-crystalline (Hyodo--Kato) cohomology of $X_{\overline{s}}$ (see \cite[\S 7.2]{FlachMorin2018}).

    The $p$-adic Weil cohomology theory for varieties $Y/k$  is rigid cohomology (with coefficients in $K_0$)
    \[
    Y \mapsto H^i_{\op{rig}}(Y)
    \]
    taking values in the category of $\vp$-modules, i.e., finite-dimensional $K_0$-vector spaces with a Frobenius-semilinear endomorphism $\vp$.
    In their article, Flach and Morin conjecture the following relationship between the two cohomology theories:
    \begin{conjecture}{
    ({\cite[Conjecture~7.15]{FlachMorin2018}})}
    For regular $X$ of absolute dimension $d=n+1$ there is an exact triangle in the category of $\vp$-modules 
\begin{align*}
    R\Gamma_{\op{rig}}(X_s) \xrightarrow{\op{sp}} \left[ R\Gamma_{HK}^B(X_{\overline{\eta},h})^{G_K} \xrightarrow{N}
    R\Gamma_{HK}^B(X_{\overline{\eta}, h})(-1)^{G_K}\right] &\xrightarrow{\op{sp}'}\\
    R\Gamma_{\op{rig}}^*(X_s)(-d)[-2d+1] &\to.
\end{align*}
    where $\op{sp}$ induces the specialization map defined in \cite{Yitao2017} and $\op{sp}'$ is the composite of the Poincar\'e duality isomorphism
    \[
        R\Gamma^B_{HK}(X_{\overline{\eta}, h})(-1) \cong R\Gamma^B_{HK}(X_{\overline{\eta},h})^*(-d)[-2d+2]
    \]
    on $X_{\overline{\eta}}$ and $\op{sp}^*$. 
    \end{conjecture}

    We can reformulate this conjecture in the case where $f$ is log-smooth  using existing comparisons between the cohomology of special and generic fibers. 
    Namely, under this extra condition, we already have a comparison between the cohomology of the special and generic fibers: Tsuji's theorem \cite[Theorem 0.2]{Tsuji1999} provides a canonical isomorphism
    \begin{align*}
    H^{B,i}_{HK}(X_{\overline{\eta}, h})^{G_K} \cong  H^i_{\op{log-crys}}(X_k/S^{\times}).
    \end{align*}
    We thereby obtain the following crystalline realization of this exact triangle in the category of $\vp$-modules:
    \begin{align}\label{flach_morin_crystalline_realization}
        R\Gamma_{\op{rig}} (X_s) \to \left[ R\Gamma_{\op{log-crys}}(X_s/S^{\times}) \xrightarrow{N} R\Gamma_{\op{log-crys}}(X_s/S^{\times})(-1)\right] \to\\
        R\Gamma^*_{\op{rig}}(X_s)(-n-1)[-2n-1] &\to. \nonumber
    \end{align}
    Thus, in the log-smooth case, the conjecture describes the monodromy operator on the log-crystalline cohomology of $X_s$ in terms of rigid cohomology and its Poincar\'e dual.

    Since the triangle (\ref{flach_morin_crystalline_realization}) no longer involves the generic fiber $X_\eta$ or the valuation ring $\O_K$, one can untwine the triangle from its original context and ask under what conditions on a $k$-scheme $X_s$ such an exact triangle exists. 
    The main result of this paper is the following theorem, which states that such an exact triangle exists when $X_s$ is the special fiber, not of an  arithmetic family $f: X \to\op{Spec}( \O_K)$, but a geometric family $f: X \to C$ where $C$ is a curve over $k$:
    \begin{theorem}\label{flach-morin_main_theorem}
    Let $f: X \to C$ be a proper, flat, generically smooth morphism over $k$ of relative dimension $n$, where $C$ is a smooth curve and $X$ is smooth.
    Assume that for some $k$-rational point $s \in C$ the fiber $X_s$ is a normal crossing divisor in $X$.
    Endow $X$ with the log structure given by the divisor $X_s$ and endow $X_s$ itself with the pullback log-structure.
    Then there is an exact triangle 
    \begin{center}
        \begin{tikzcd}[row sep=4em, column sep=-3.2em]
           & R\Gamma^*_{\op{rig}}(X_s)(-n-1)[-2n-1] \ar[dl]\\
        R\Gamma_{\op{rig}} (X_s) \ar[rr] && \left[ R\Gamma_{\op{log-crys}}(X_s/S^{\times}) \xrightarrow{N} R\Gamma_{\op{log-crys}}(X_s/S^{\times})(-1)\right] \ar[ul]
        \end{tikzcd}
    \end{center}
    in the derived category of $\vp$-modules.
    \end{theorem}

    Our driving methodology is to adapt Chiarellotto and Tsuzuki's proof \cite{ChiarellottoTsuzuki2014} of the existence and exactness of the Clemens--Schmid sequence
    \begin{align*}
        \cdots &\to H^m_{\op{rig}}(X_s) \xrightarrow{\gamma} H^m_{\op{log-crys}}((X_s,M_s)/S^{\times}) \otimes K \xrightarrow{N_m}\\
        &H^m_{\op{log-crys}}((X_s,M_s)/S^{\times}) \otimes K(-1) \xrightarrow{\delta} H^{m+2}_{X_s,\op{rig}}(X) \xrightarrow{\alpha} H^{m+2}_{\op{rig}}(X_s) \to \cdots
    \end{align*}
    together with the theory of log-convergent and log-rigid cohomology. In \cite{ChiarellottoTsuzuki2014} the authors need to restrict themselves to the case of finite fields to prove their main result because they use \cite{Crew98}. Since in the present article we do not need to compare the monodromy and the weight filtrations, we can avoid this restriction. Another difference from \cite{ChiarellottoTsuzuki2014} is that we prove a result at the level of the derived category of complexes and not just in cohomology.

    The similarity between the Clemens--Schmid sequence and the Flach-Morin triangle is clear in light of Berthelot's Poincar\'e duality \cite{Berthelot1997}, which states that the dual of rigid cohomology is given by the rigid cohomology with compact support
    \[
        R\Gamma_{X_s,\op{rig}}(Y) \cong R\Gamma^*_{\op{rig}}(X_s)(-n-1)[-2n-2]
    \]
    where $Y$ is a smooth $k$-scheme admitting a closed immersion $X_s \hookrightarrow Y$; in our context we may choose, in particular, $Y = X$.

    We follow their general idea of linking the localization triangle for rigid cohomology with respect to the closed subscheme $X_s$ with the canonical exact triangle for the monodromy operator in log-crystalline cohomology, using the fact that the rigid cohomology of the open complement of a closed subscheme can be computed using logarithmic structures.

    We work with log-rigid cohomology (see \S~\ref{notation}) in place of log-crystalline cohomology because of the flexibility of the former.
    This is possible because log-crystalline and log-rigid cohomology, and their respective monodromy operators, agree in the proper and log-smooth case (Lemma~\ref{lmm:comp}).

After the submission of this article, Binda--Gallauer--Vezzani showed in \cite{BGV23} how to deduce the Clemens--Schmid sequence and the Flach--Morin conjecture with motivic methods. While their approach proves the general case, our method is more explicit. \\[1ex]
\textbf{Acknowledgments.} The authors thank Veronika Ertl for insightful conversations. They also warmly thank the referee for the valuable remarks that significantly improved the quality of the article.
\section{Review of Rigid Cohomology}\label{notation}
    We begin by reviewing the basic concepts of rigid cohomology that we will need in the sequel.
    For the moment, we forget about log structures.
    A \emph{frame} $(X \subseteq Y \subseteq \mathfrak{P})$ (see \cite[Definition 3.1.5]{LeStum2007}) is a sequence of inclusions
    \[
    X \hookrightarrow Y \hookrightarrow \mathfrak{P}    
    \]
    where $X \hookrightarrow Y$ is an open immersion of $k$-varieties and $Y \hookrightarrow \mathfrak{P}$ is a closed immersion in a formal $W$-scheme $\mathfrak{P}$. 

    Fix a frame $(X \subseteq X' \subseteq \mathfrak{P})$ where $X,X'$, and $\mathfrak{P}$ are separated and locally of finite type and where $\mathfrak{P}$ is smooth in a neighborhood of $X$. 
    The \emph{rigid cohomology of the pair $(X,X')$} is defined as
    \begin{equation}\label{eq:rigid}
           R\Gamma_{\op{rig}}((X,X')) := R\Gamma(\tube{X'}{\fP}, \jdag_{X}\Omega^\bullet_{\tube{X'}{\fP}})\ ,
    \end{equation}
    where $\jdag_{X}$ is the functor of overconvergent sections \cite[p.129]{LeStum2007} and is denoted by $\jdag_{\tube{X}{\fP}}$ in \cite{ChiarellottoTsuzuki2014}. In general, for a rigid analytic variety $V$ and an admissible open $T\subset V$ it is possible to define the functor $\jdag_{V\setminus T}$ of overconvergent sections along $T$. When $V\subset \tube{X'}{\fP}$ is a strict neighborhood of $\tube{X}{\fP}$ and $T=V\cap (\tube{X'}{\fP}\setminus \tube{X}{\fP})$ we have $\jdag_X=\jdag_{V\setminus T}$.
    \emph{A priori} the definition in \eqref{eq:rigid} also depends on the formal scheme $\mathfrak{P}$, but it can be shown (see \cite[7.4.2]{LeStum2007}, for example) that it depends (up to quasi-isomorphism) only on the immersion $X \hookrightarrow X'$.

    The two important cases of this construction that we use are the following.
    First, the \emph{convergent cohomology} of $X$ is defined to be 
    \[
    R\Gamma_{\op{conv}}(X) := R\Gamma_{\op{rig}}((X,X)).
    \]
    Second, if $X \hookrightarrow \overline{X}$ is an open immersion with $\overline{X}$ proper, the \emph{rigid cohomology} of $X$ is defined to be 
    \[
        R\Gamma_{\op{rig}}(X) := R\Gamma_{\op{rig}}((X,\overline{X})).
    \]
    It is a fundamental result of the rigid cohomology theory established by Berthelot that this definition is not only independent of the formal scheme $\mathfrak{P}$, but also of the compactification $\overline{X}$ (see \cite[Proposition 8.2.1]{LeStum2007}).
    Note that if $X$ is already proper, then the convergent cohomology coincides with the rigid cohomology.

    To understand the dual $R\Gamma^*_{\op{rig}}(X_s)$, we will also need the notion of rigid cohomology with support in a closed subset.
    \begin{definition}\label{rigid_coh_with_support_def}
        Let $X$ be a $k$-scheme, $Z \subseteq X$ a closed subscheme, and fix a frame $(X \subseteq \overline{X} \subseteq \fP)$ where $\overline{X}$ is proper and $\fP$ is smooth in a neighborhood of $X$.
        We define 
        \[
        R\Gamma_{Z,\op{rig}}(X) := R\Gamma(\tube{\overline{X}}{P},\underline{\Gamma}^\dagger_{\tube{Z}{\fP}} \jdag_X\Omega^\bullet_{\tube{\overline{X}}{\fP}})
        \]
        to be the \emph{rigid cohomology of $X$ with support in $Z$}. (see \cite{Berthelot1997} or \cite[Definition 6.3.1]{LeStum2007})
    \end{definition}   

    Its relation to standard rigid cohomology is as follows.
    As an immediate consequence of the definition (see \cite[Proposition 5.2.4 (ii)]{LeStum2007}) we have, for any sheaf $\cE$ on $\tube{\overline{X}}{\fP}$, the following exact sequence 
    \[
        0 \to \uGamma^\dagger_Z \jdag_X \cE \to \jdag_X \cE \to \jdag_{\tube{\overline{X}}{\fP} \setminus \tube{Z}{\fP}}\jdag_X \cE \to 0\ .
    \]
    But by \cite[Proposition 5.1.7]{LeStum2007}, we get the exact sequence
    \[
        0 \to \uGamma^\dagger_Z \jdag_X \cE \to \jdag_X \cE \to \jdag_U \cE \to 0  \ ,
    \]
      where $U=X\setminus Z$.
    In this way we obtain an exact sequence 
    \[
    0 \to \underline{\Gamma}^\dagger_{\tube{Z}{\fP}} \jdag_X\Omega^\bullet_{\tube{\overline{X}}{\fP}} \to \jdag_X\Omega^\bullet_{\tube{\overline{X}}{\fP}} \to \jdag_U\Omega^\bullet_{\tube{\overline{X}}{\fP}} \to 0
    \]
    which induces in cohomology the \emph{localization triangle}
    \[
    R\Gamma_{Z,\op{rig}}(X) \to R\Gamma_{\op{rig}}(X) \to R\Gamma_{\op{rig}}(U) \xrightarrow{+}.  
    \]


    There is a theory of log-rigid cohomology, described by Gro{\ss}e-Kl\"onne (see, for example, \cite{GrosseKlonne2005}) that generalizes the log-crystalline cohomology of Hyodo and Kato in the same way that rigid cohomology generalizes crystalline cohomology.
    Here, we give a brief overview and refer the reader to \cite[\S 1.3]{GrosseKlonne2005} for precise definitions.

    Write $\fS := \op{Spf}(W)$ and let $\fS^{\times}$ (resp. $\fS^\varnothing$) denote the weak formal log-scheme $(\fS, 1 \mapsto 0)$ (resp. $(\fS, \text{trivial})$).
    Let $X$ be a fine log-scheme over the log point $s^\times=(\op{Spec}(k), 1 \mapsto 0)$.
    One can choose an open covering $X = \bigcup_{i \in I}V_i$ with exact closed immersions $V_i \hookrightarrow \mathfrak{V}_i$ and for each $H \subseteq I$ let $\mathfrak{V}_H$ be  an exactification 
    \[
    V_H := \bigcap_{i \in H}V_i \to \mathfrak{V}_H \to  \varprojlim_{i\in H}  \mathfrak{V}_i \ ,
    \]
    where the inverse limit is taken in the category of weak formal schemes over $\fS$.
    From these weak formal embeddings we can canonically construct a simplicial dagger space $\tube{V_\bullet}{\mathfrak{V}_\bullet} := (\tube{V_H}{\mathfrak{V}_H})_{H \subseteq I}$. 
    The cohomology of the corresponding de Rham complex
    \[
    R\Gamma_{\op{log-rig}}(X/\fS^{\times}) := R\Gamma(\tube{V_\bullet}{\mathfrak{V}_\bullet}, \Omega^\bullet_{\tube{V_\bullet}{\mathfrak{V}_\bullet}})
    \]
    is defined as the \emph{log-rigid cohomology of $X$} with respect to $\fS^{\times}$.
    It can be shown to be independent of the covering $X = \bigcup_i V_i$.\footnote{Although we will not use it, it might be of interest for the reader that  the complex $R\Gamma_{\op{log-rig}}(X_s/\fS^{\times})$ together with its monodromy operator can be calculated using the overconvergent logarithmic de Rham--Witt complex of Gregory--Langer \cite{GL2020}.  }

    One can also define a logarithmic equivalent of convergent cohomology for a fine log scheme $X/k$, called log-convergent cohomology, which we denote by $R\Gamma_{\op{log-conv}}(X/\fS^\times)$.
    It is generally defined formally as the cohomology of the trivial isocrystal on the log-convergent site (\cite[\S 2.1]{Shiho2002}), but it also admits an interpretation through the cohomology of a suitable logarithmic de Rham complex \cite[\S2.1, Corollary 2.3.9]{Shiho2002}. 
    In fact, by replacing weak formal schemes and dagger spaces with formal schemes and rigid spaces in the definition of log-rigid cohomology, one recovers Shiho's log-convergent cohomology \cite[\S 1.5]{GrosseKlonne2005}.
    In particular, if $X$ is proper, then log-convergent cohomology coincides with log-rigid cohomology.
    If $X$ is additionally log-smooth, these two cohomologies also coincide with log-crystalline cohomology \cite[pp.\,401]{GrosseKlonne2005}.

    Finally, we can also repeat the construction of log-rigid cohomology with $\fS^\varnothing$ to obtain cohomology groups $R\Gamma_{\op{log-rig}}(X/\fS^\varnothing)$. 
    If now $X$ is a $k$-scheme equipped with the trivial log structure, not necessarily proper, we have an isomorphism (see \cite[pp.\,401]{GrosseKlonne2005})
    \[
    R\Gamma_{\op{log-rig}}(X/\fS^\varnothing) \cong R\Gamma_{\op{rig}}(X).     
    \]

    \section{Proof of the Main Result}
    In this section we prove Theorem~\ref{flach-morin_main_theorem}, stated in the Introduction.
    \begin{remark}
        Our setting is the derived category of $\vp$-modules, and it will be implicit that all the quasi-isomorphisms below are compatible with Frobenius when the objects have a non-trivial structure of the $\vp$-module.
    \end{remark}
    \begin{lemma}\label{lmm:comp}
        To prove Theorem \ref{flach-morin_main_theorem}, it suffices to construct an exact triangle as in the statement where $R\Gamma_{\op{log-crys}}(X_s/S^{\times})$ and its monodromy operator $N$ are replaced by $R\Gamma_{\op{log-rig}}(X_s/\fS^{\times})$ and its monodromy operator. 
    \end{lemma}
    \begin{proof}
        By \cite[Theorem 3.1.1]{Shiho2002} and \cite[\S 1.5]{GrosseKlonne2005}, there is a map
        \[
        \alpha: R\Gamma_{\op{log-rig}}(X_s/\fS^{\times}) \to  R\Gamma_{\op{log-crys}}(X_s/S^{\times}) \ 
        \]
        which is an isomorphism because we are in the proper and log-smooth case. 
        The map $\alpha$ is obtained by considering that the log-rigid cohomology coincides with the log-convergent  one and the latter can be calculated by classical tubes when the immersion is exact. Also, there is a natural map from these classical tubes to the divided-power ones which are used to calculate the log-crystalline cohomology (cf. \cite[Proof of proposition~1.9]{BerFinitude}). In both cases, monodromy operators are defined as the connecting morphism arising from appropriate short exact sequences coming from an embedding system, which can be used for both cohomology theories:  we can assume that the embedding system is exact, for instance as in \cite[\S~5.2]{GrosseKlonne2005}. The log-rigid one is detailed in \cite[\S 5.4]{GrosseKlonne2005} and the log-crystalline one is analogous.
    \end{proof}
 \begin{remark}[log-rigid comparison]\label{log-rig_instead_of_log-crys}
    Note that log-rigid cohomology over $\fS^{\times}$ can alternatively be defined as \cite{ErtlYamadaRigAnalRecon}. The definition of Ertl--Yamada gives cohomology groups isomorphic to those of Gro{\ss}e-Kl\"onne \cite[Remark 2.4]{ErtlYamadaRigAnalRecon}. In both cases, the monodromy operator is defined starting with a short exact sequence \cite[eq. (3.36) and (3.37)]{ErtlYamadaRigAnalRecon} and \cite[\S~5.4]{GrosseKlonne2005} and there is a natural map between the two.

    In fact, in \cite{ErtlYamadaRigAnalRecon} the authors use an embedding system defined over $W[[s]]$, while in \cite{GrosseKlonne2005} it is over $W[s]^\dag$. The embedding system can be constructed as in \cite[Lemma~2.6 and Definition~2.7]{ErtlYamadaRigAnalRecon} both over $W[s]^\dag$ and over $W[[s]]$, compatibly with the natural map $W[s]^\dag\to W[[s]]$. From this we get the map of short exact sequences and 
 the monodromy is compatible because of isomorphism between log-rigid cohomology groups over $\fS^{\times}$.
    \end{remark}

As a first step, we reformulate the results of \cite{ChiarellottoTsuzuki2014} entirely in the language of derived categories. Note that the open inclusion $X\setminus X_s\to X$ induces a natural map $\iota:R\Gamma_{\op{rig}} (X,X)\to R\Gamma_{\op{rig}} (X\setminus X_s,X)$. Moreover $R\Gamma_{\op{rig}} (X,X)=R\Gamma_{\op{conv}}(X)$.
First we prove 
\begin{lemma}\label{flach_morin_support_no_compactification}
    There is a canonical isomorphism 
    \[
    R\Gamma_{X_s, \op{rig}}(X) \cong [\iota:R\Gamma_{\op{conv}}(X) \to R\Gamma_{\op{rig}}((X \setminus X_s, X))]
    \]
    where $[-]$ denotes the homotopy limit.
\end{lemma}
In other words, $R\Gamma_{X_s,\op{rig}}(X)$ can be computed without passing to a compactification.

\begin{proof}
    For our basic landscape to compute rigid cohomology, we fix, as in \cite[\S 4]{ChiarellottoTsuzuki2014}, a simplicial Zariski hypercovering of the form 
    \begin{center}
        \begin{tikzcd}
            X_{s,\bullet} \ar[r] \ar[d] & X_\bullet \ar[r] \ar[d] & \overline{X}_\bullet \ar[r] \ar[d] & \mathscr{P}_\bullet\\
            X_s \ar[r] \ar[d] & X \ar[r] \ar[d] & \overline{X} \\ 
            s \ar[r] & C.
        \end{tikzcd}
    \end{center}
    Here the simplicial map $\overline{X}_\bullet \to \overline{X}$ is a Zariski affine hypercovering, $\mathscr{P}_\bullet$ is a simplicial formal scheme separated and of finite type over $\O_K$ which is smooth around $X_\bullet$, and which admits a Frobenius $\sigma_\bullet$ lifting that on $\O_K$\footnote{In \cite[\S 4]{ChiarellottoTsuzuki2014} the authors also need a compactification $\overline{X}$ over a smooth compactification of $C$ to apply a result of Crew on weight-monodromy in positive characteristic. In the present paper we do not use the result of Crew and hence we do not need such a compactification. }.
    Then by definition, we have 
    \begin{equation}\label{flach_morin_rigid_cohomology_def}
    R\Gamma_{X_s,\op{rig}}(X) = R\Gamma(\tube{\overline{X}_\bullet}{\mathscr{P_\bullet}},(\jdag_{\tube{X_\bullet}{\mathscr{P}_\bullet}}\Omega^\bullet_{\tube{\overline{X}_\bullet}{\mathscr{P}_\bullet}} \to \jdag_{\tube{X_\bullet \setminus X_{s,\bullet }}{\mathscr{P}_\bullet}}\Omega^\bullet_{\tube{\overline{X}_\bullet}{\mathscr{P}_\bullet}})_s)    
    \end{equation}
    where $(-)_s$ denotes the total complex of the morphism of complexes, interpreted as the rows of a double complex with $\jdag_{\tube{X_\bullet}{\mathscr{P}_\bullet}}\O_{\tube{\overline{X}_\bullet}{\mathscr{P}_\bullet}}$ in degree $(0,0)$. Here, for $U_\bullet= X_\bullet$ or $X_\bullet \setminus X_{s,\bullet }$, the symbol $\jdag_{\tube{U_\bullet}{\mathscr{P}_\bullet}}$ denotes the functor of overconvergent sections along $\tube{\overline{X}_\bullet \setminus U_\bullet}{\mathscr{P}_\bullet}$

    Consider now the following admissible covering of $\tube{\overline{X}_\bullet}{\mathscr{P}_\bullet}$
    \[
    \{
        \tube{\overline{X}_\bullet}{\mathscr{P}_\bullet} \setminus \tube{X_{s,\bullet}}{\mathscr{P}_\bullet}, \tube{X_\bullet}{\mathscr{P}_\bullet}.
        \}    
    \]
    Note that the two constituent complexes in (\ref{flach_morin_rigid_cohomology_def}) agree on the former admissible open subset and also on the intersection $(\tube{\overline{X}_\bullet}{\mathscr{P}_\bullet} \setminus \tube{X_{s,\bullet}}{\mathscr{P}_\bullet}) \cap \tube{X_\bullet}{\mathscr{P}_\bullet}$.
    As such the total complex in (\ref{flach_morin_rigid_cohomology_def}) restricts, on these admissible opens, to the total complex of the identity map, which has trivial cohomology \cite[Exercise 1.5.1]{Weibel1994}.
    It follows from Zariski descent that 
    \begin{align*}
        R\Gamma_{X_s,\op{rig}}(X) &= R\Gamma(\tube{\overline{X}_\bullet}{\mathscr{P_\bullet}},(\jdag_{\tube{X_\bullet}{\mathscr{P}_\bullet}}\Omega^\bullet_{\tube{\overline{X}_\bullet}{\mathscr{P}_\bullet}} \to \jdag_{\tube{X_\bullet \setminus X_{s,\bullet }}{\mathscr{P}_\bullet}}\Omega^\bullet_{\tube{\overline{X}_\bullet}{\mathscr{P}_\bullet}})_s) \\
        &\cong R\Gamma(\tube{\overline{X}_\bullet}{\mathscr{P_\bullet}},Ra_{\bullet, *}((\jdag_{\tube{X_\bullet}{\mathscr{P}_\bullet}}\Omega^\bullet_{\tube{\overline{X}_\bullet}{\mathscr{P}_\bullet}})|_{\tube{X_\bullet}{\mathscr{P}_\bullet}}\\
        &\hspace{120pt} \to (\jdag_{\tube{X_\bullet \setminus X_{s,\bullet }}{\mathscr{P}_\bullet}}\Omega^\bullet_{\tube{\overline{X}_\bullet}{\mathscr{P}_\bullet}})|_{\tube{X_\bullet}{\mathscr{P}_\bullet}})_s)
    \end{align*}
    where $a_{\bullet}: \tube{X_\bullet}{\mathscr{P}_\bullet} \hookrightarrow \tube{\overline{X}_\bullet}{\mathscr{P}_\bullet}$ is the inclusion.

    Furthermore, it was shown in \cite[Proposition 4.1]{ChiarellottoTsuzuki2014} that the direct image $a_{\bullet, *}$ is exact on the sheaves $\Omega^\bullet_{\tube{X_\bullet}{\mathscr{P}_\bullet}}$ and $\jdag_{\tube{X_\bullet \setminus X_{s,\bullet }}{\mathscr{P}_\bullet}}\Omega^\bullet_{\tube{X_\bullet}{\mathscr{P}_\bullet}}$, resulting from the fact that $X_s$ is a divisor.
    Hence
    \begin{align*}
        R\Gamma_{X_s,\op{rig}}(X) &= R\Gamma(\tube{\overline{X}_\bullet}{\mathscr{P_\bullet}},(\jdag_{\tube{X_\bullet}{\mathscr{P}_\bullet}}\Omega^\bullet_{\tube{\overline{X}_\bullet}{\mathscr{P}_\bullet}} \to \jdag_{\tube{X_\bullet \setminus X_{s,\bullet }}{\mathscr{P}_\bullet}}\Omega^\bullet_{\tube{\overline{X}_\bullet}{\mathscr{P}_\bullet}})_s) \\
        &\cong R\Gamma(\tube{\overline{X}_\bullet}{\mathscr{P_\bullet}},Ra_{\bullet, *}((\jdag_{\tube{X_\bullet}{\mathscr{P}_\bullet}}\Omega^\bullet_{\tube{\overline{X}_\bullet}{\mathscr{P}_\bullet}})|_{\tube{X_\bullet}{\mathscr{P}_\bullet}}\\ 
        &\hspace{120pt} \to (\jdag_{\tube{X_\bullet \setminus X_{s,\bullet }}{\mathscr{P}_\bullet}}\Omega^\bullet_{\tube{\overline{X}_\bullet}{\mathscr{P}_\bullet}})|_{\tube{X_\bullet}{\mathscr{P}_\bullet}})_s)\\
        &\cong R\Gamma(\tube{\overline{X}_\bullet}{\mathscr{P}_\bullet}, a_{\bullet,*}(\Omega^\bullet_{\tube{X_\bullet}{\mathscr{P}_\bullet}} \to \jdag_{\tube{X_\bullet \setminus X_{s,\bullet }}{\mathscr{P}_\bullet}}\Omega^\bullet_{\tube{X_\bullet}{\mathscr{P}_\bullet}})_s)\\ 
        &\cong R\Gamma(\tube{X_\bullet}{\mathscr{P}_\bullet},(\Omega^\bullet_{\tube{X_\bullet}{\mathscr{P}_\bullet}} \to \jdag_{\tube{X_\bullet \setminus X_{s,\bullet }}{\mathscr{P}_\bullet}}\Omega^\bullet_{\tube{X_\bullet}{\mathscr{P}_\bullet}})_s)\\
        &= [R\Gamma_{\op{conv}}(X) \to R\Gamma_{\op{rig}}((X \setminus X_s, X))].
    \end{align*}
    In the final line we've used the fact that the total complex of a morphism $\left(A^\bullet \to B^\bullet\right)_s$ is a representative of its mapping fiber $[A^\bullet \to B^\bullet]$.
    This is what we wanted to prove.
\end{proof}

    Now that we've established that the cohomology $R\Gamma_{X_s,\op{rig}}(X)$ can be computed without a compactification of (a simplicial cover of) $X$, we are now able to compute $R\Gamma_{X_s,\op{rig}}(X)$ via an alternative covering which replaces the data of a good compactification of $X$ with a covering respecting the log structure on $X$.

    Denote by $M$ the log structure on $X$ associated to the normal crossing divisor $X_s$.  In the following, the scheme $X_s$ will be considered as a log-scheme with  log-structure  induced by $M$.

\begin{lemma}\label{flach_morin_conv_logrig_Xs}
     There is a canonical isomorphism 
    \[
        R\Gamma_{X_s, \op{rig}}(X) \cong [R\Gamma_{\op{conv}}(X_s) \to R\Gamma_{\op{log-rig}}(X_s/\fS^\varnothing)].
    \]    
\end{lemma}

\begin{proof}
    Denote by $M$ the log structure on $X$ associated to the normal crossing divisor $X_s$.
    We also attach to $C$ the log structure $N$ associated to the closed point $s \in C$, so that in particular the morphism $(X,M) \to (C,N)$ is log smooth. 

    Because $C$ is a smooth curve over $k$, it admits a smooth lifting $\widetilde{C}$ over $\O_K$ \cite[Corollaire III.7.4]{SGAI}.
    Let $\mathscr{C}$ be the completion of $\widetilde{C}$ along its special fiber $C$, let $\hat{s}$ be a lift of $s$ to $\mathscr{C}$ and let $t$ be a local coordinate of $\hat{s}$ over $\O_K$.
    Then $1 \mapsto t$ defines a log structure on $\mathscr{C}$, which we denote by $\mathscr{N}$.
    Then the maps
    \[
    s^\times \to (C,N) \to (\mathscr{C}, \mathscr{N})    
    \]
    are exact closed immersions and the latter two are log smooth over $(\op{Spec}k)^\varnothing$ and $\fS^\varnothing$, respectively.

    By \cite[Proposition 4.3]{ChiarellottoTsuzuki2014}, we can construct a simplicial \'etale hypercovering
    \begin{center}
        \begin{tikzcd}
            (X_\bullet, M_\bullet) \ar[r, "i_\bullet^{\op{ex}}"] \ar[d] & (\widetilde{\mathcal{Q}}_\bullet^{\op{ex}},\widetilde{\mathscr{M}}_\bullet) \ar[dd]\\
            (X,M) \ar[d] & \\
            (C,N) \ar[r] & (\mathscr{C}, \mathscr{N})
        \end{tikzcd}
    \end{center}
    where the log structure $M_\bullet$ on $X_\bullet$ is that induced by $M$ and where $i_\bullet^{\op{ex}}$ is an exact closed immersion into a simplicial log-formal scheme $(\widetilde{\mathcal{Q}}_\bullet^{\op{ex}},\widetilde{\mathscr{M}}_\bullet)$ which is log-smooth over $(\sC, \mathscr{N})$ and with $\widetilde{\mathcal{Q}}_{\bullet}^{\op{ex}}$ separated, smooth, and of finite type over $\O_K$.
    Moreover $(\widetilde{\mathcal{Q}}^{\op{ex}}_\bullet, \widetilde{\mathscr{M}}_\bullet) \to (\mathscr{C},\mathscr{N})$ is formally log-smooth and admits a lift of Frobenius.

    \begin{remark}
    The virtue of such a hypercovering is that, as we will see, it can be used both to compute rigid cohomology \emph{and} logarithmic rigid cohomology. 
    \end{remark}

    If we let $X_{s,\bullet}$ denote the induced \'etale hypercovering of $X_s$, we can now apply a pair of results of Shiho \cite[Corollary 2.3.9, Proposition 2.4.4]{Shiho2002} which in conjunction say that we have an identification 
    \[R\Gamma_{\op{rig}}((X_m \setminus X_{s,m}, X_m)) \cong R\Gamma_{\op{log-conv}}((X_m,M_m)/\fS^\varnothing)\ ,
    \]
    because every $(X_m,M_m)$ is a smooth log scheme and its log-structure is  trivial on $X_m \setminus X_{s,m}$.

    It follows by \'etale descent on rigid cohomology \cite{ChiarellottoTsuzuki2003} that 
    \[
        R\Gamma_{\op{rig}}((X\setminus X_s, X)) \cong R\Gamma_{\op{log-conv}}((X,M)/\fS^\varnothing).
    \]
    It is clear by looking at  their representative de Rham complexes that this isomorphism fits into the commutative diagram
    \begin{center}
        \begin{tikzcd}
            R\Gamma_{\op{conv}}(X) \ar[r] \ar[d, "="] & R\Gamma_{\op{log-conv}}((X,M)/\fS^\varnothing) \ar[d, "\cong"]\\
            R\Gamma_{\op{conv}}(X) \ar[r] & R\Gamma_{\op{rig}}((X \setminus X_s, X))
        \end{tikzcd}
    \end{center}
    so to prove the lemma it suffices to show, by Lemma \ref{flach_morin_support_no_compactification}, that
    \begin{equation*}
        \scalebox{.92}{$[R\Gamma_{\op{conv}}(X) \to R\Gamma_{\op{log-conv}}((X,M)/\fS^\varnothing)] \cong [R\Gamma_{\op{conv}}(X_s) \to R\Gamma_{\op{log-conv}}((X_s,M_s)/\fS^\varnothing)]$}.
    \end{equation*}

    The argument is similar to that of Lemma \ref{flach_morin_support_no_compactification}.
    The former mapping fiber can be written as
    \begin{equation}\label{flach_morin_conv_logconv_X}
        R\Gamma(\tube{X_\bullet}{\widetilde{\cQ}^{\op{ex}}_\bullet},
    (\Omega^\bullet_{\tube{X_\bullet}{\widetilde{\cQ}^{\op{ex}}_\bullet}} \to \Omega^\bullet_{\tube{X_\bullet}{\widetilde{\cQ}^{\op{ex}}_\bullet}} \langle \widetilde{\sM_\bullet}\rangle)_s).
    \end{equation}
    To compute this, consider the admissible cover of $\sC_\eta$ given by the tube of $\{s\}$ in $\sC$ and $V$ a strict affinoid neighborhood of $C \setminus \{s\}$ in $\sC$.
    Their inverse image in $\tube{X_\bullet}{\widetilde{\cQ}^{\op{ex}}_\bullet}$ is an admissible covering as well; the inverse image of the tube of $\{s\}$ is given by $\tube{X_{s,\bullet}}{\widetilde{\cQ}^{\op{ex}}_\bullet}$ and we denote by $V_\bullet$ the inverse image of $V$.
    The restrictions to $V_\bullet$ of the two complexes in \eqref{flach_morin_conv_logconv_X}  are the same, and the inclusion $\iota_{\bullet, \eta}: \tube{X_{s,\bullet}}{\widetilde{\cQ}^{\op{ex}}_\bullet} \to \tube{X_\bullet}{\widetilde{\cQ}^{\op{ex}}_\bullet}$ is quasi-Stein, again since $X_s$ is a divisor.
    It thus follows by the same argument as in Lemma \ref{flach_morin_support_no_compactification}, using Kiehl's result that $R\iota_{\bullet, \eta *} = \iota_{\bullet, \eta *}$ on coherent sheaves \cite[Satz~2.4]{Kie67}, that (\ref{flach_morin_conv_logconv_X}) is isomorphic to
    \[
        R\Gamma(\tube{X_{s,\bullet}}{\widetilde{\cQ}^{\op{ex}}_\bullet},
    (\Omega^\bullet_{\tube{X_{s,\bullet}}{\widetilde{\cQ}^{\op{ex}}_\bullet}} \to \Omega^\bullet_{\tube{X_{s,\bullet}}{\widetilde{\cQ}^{\op{ex}}_\bullet}} \langle \widetilde{\sM_\bullet}\rangle)_s)
    \]
    But this in turn is isomorphic to 
    \[
        [R\Gamma_{\op{conv}}(X_s) \to R\Gamma_{\op{log-conv}}(X_s/\fS^\varnothing)]    
    \]
    as desired.
\end{proof}

We are now in a position to prove the main claim:

\begin{proof}[Proof of Theorem \ref{flach-morin_main_theorem}]
    Due to Lemma~\ref{lmm:comp} we prove the statement with ``log-rig'' in place of ``log-crys'' and $R\Gamma_{\op{log-rig}}(X_s/\fS^{\times})$ is the log-rigid cohomology complex of Ertl--Yamada (see Remark~\ref{log-rig_instead_of_log-crys}).
   
    By \cite[Proposition 3.33~(1)]{ErtlYamadaRigAnalRecon}, the associated exact triangle is
    \[
    R\Gamma_{\op{log-rig}}(X_s/\fS^\varnothing) \to R\Gamma_{\op{log-rig}}(X_s/\fS^{\times}) \xrightarrow{N} R\Gamma_{\op{log-rig}}(X_s/\fS^{\times})(-1) \to.
    \]
    (The twist $R\Gamma_{\op{log-rig}}(X_s/\fS^{\times})(-1)$ arises from the fact that, as expected, $N\vp = p \vp N$; see \cite[Proposition 5.5]{GrosseKlonne2005}).
    We thus obtain a quasi-isomorphism
    \[
    R\Gamma_{\op{log-rig}}(X_s/\fS^\varnothing) \cong [R\Gamma_{\op{log-rig}}(X_s/\fS^{\times}) \xrightarrow{N} R\Gamma_{\op{log-rig}}(X_s/\fS^{\times})(-1)]. 
    \]
    
    Plugging this into the exact triangle corresponding to the mapping fiber in Lemma \ref{flach_morin_conv_logrig_Xs} we obtain an exact triangle 
    \begin{align*}
    R\Gamma_{X_s, \op{rig}}(X) \to R\Gamma_{\op{conv}}&(X_s) \to\\ 
    & [R\Gamma_{\op{log-rig}}(X_s/\fS^{\times}) \xrightarrow{N} R\Gamma_{\op{log-rig}}(X_s/\fS^{\times})(-1)] \to. 
    \end{align*}

    Finally, Poincar\'e duality \cite[Th\'eor\`eme 2.4]{Berthelot1997} (see \cite[\S 2.1]{ChiarellottoLeStum1999} for details on the Frobenius action) provides a canonical isomorphism
    \[
    R\Gamma_{X_s,\op{rig}}(X) \cong R\Gamma_{\op{rig}}(X_s)^*(-n-1)[-2n-2]
    \]
    so after substitution and shifting the triangle we obtain an exact triangle 
    \begin{center}
        \begin{tikzcd}[row sep=4em, column sep=-3.5em]
            & R\Gamma_{\op{rig}}(X_s)^*(-n-1)[-2n-1] \ar[dl] \\
        R\Gamma_{\op{rig}}(X_s) \ar[rr] && \left[R\Gamma_{\op{log-rig}}(X_s/\fS^{\times}) \xrightarrow{N} R\Gamma_{\op{log-rig}}(X_s/\fS^{\times})(-1)\right] \ar[ul]
        \end{tikzcd}
    \end{center}

    of $\vp$-modules, as desired.
\end{proof}

\begin{remark}
Of course a scheme $Y$ in characteristic $p$ can be embedded as a closed subscheme in other ways and these other embeddings suggest further directions for research.
For example, $Y$ may be the fiber of a 1-dimensional \emph{arithmetic} family, such as a discrete valuation ring of mixed or equal characteristic.

It may be interesting to consider $Y$ as as the special fiber of a scheme $X$ over  a complete discrete valuation ring with residue field $k$ and to give meaning to the cohomology
\[
    R\Gamma_Y(X)
\]
with support in the special fiber.
It may be possible to extend the proof to this situation if such a cohomology is defined.
One could also study a family over a discrete valuation ring of equicharacteristic $p$, for example when $Y$ is the special fiber of a scheme $X$ over $k\llbracket t \rrbracket$.
Rigid cohomology over such a Laurent series has a more analytic flavor.
It is defined to be a functor
\[
X \mapsto H^*_{\op{rig}}(X/\mathscr{E}_K)    
\]
taking values in graded vector spaces over the Amice ring 
\[
\mathscr{E}_K:= \left\{ \sum_{i \in \Z}a_it^i \in K\llbracket t,t^{-1}\rrbracket \,:\, \op{sup}_i|a_i|<\infty, \lim_{i \to -\infty} a_i = 0 \right\}.
\]
To study these objects one can study instead rigid cohomology over the bounded Robba ring $X \mapsto H^*_{\op{rig}}(X/\mathscr{E}^\dagger_K)$, where
\[
    \mathscr{E}^\dagger_K = \left\{ \sum_{i \in \Z}a_it^i \in K\llbracket t,t^{-1}\rrbracket : \op{sup}_i |a_i|<\infty, \exists \eta<1 \op{s.t.} \lim_{i \to -\infty} |a_i|\eta^i = 0\right\}    
\]
The bounded Robba ring has the additional virtue that it is a Henselian discretely valued field with residue field $k((t))$.
This cohomology theory is constructed so that when we base change to $\mathscr{E}_K$ one recovers $\mathscr{E}_K$-valued rigid cohomology (see \cite[\S 2.2]{LazdaPal2016} for details).
This is a direction for further research.
\end{remark}
    \bibliographystyle{halpha-abbrv}


\end{document}